\documentclass[11pt,reqno]{amsart}

\usepackage{amsopn,amstext,amsbsy,amsmath,amscd,amsthm,amsfonts}

\usepackage{marvosym}

\usepackage{color}
\usepackage{amsfonts,amscd}
\usepackage{amssymb}
\usepackage{url}
\usepackage[english]{babel}
\usepackage{hyperref}

\theoremstyle{plain}
\newtheorem{theorem}                 {Theorem}      [section]

\newtheorem{corollary}    [theorem]  {Corollary}
\newtheorem{lemma}        [theorem]  {Lemma}
\newtheorem{proposition}  [theorem]  {Proposition}

\theoremstyle{definition}

\newtheorem{definition}   [theorem]  {Definition}

\newtheorem{example}      [theorem]  {Example}

\newcommand{\eg}[1]{ {\textcolor{red}{#1}}}

\numberwithin{equation}{section}

\def \cn{{\mathbb C}}

\def \H{{\mathbb H}}

\def \rn{{\mathbb R}}

\def \H{\mathcal H}

\def \V{\mathcal V}

\def\nab#1#2{\hbox{$\nabla$\kern -.3em\lower 1.0 ex
		\hbox{$#1$}\kern -.1 em {$#2$}}}

\def \ip #1#2{\langle #1,#2 \rangle}
\def \lb#1#2{[#1,#2]}

\def \r{\mathfrak{r}}

\def\jid(#1#2#3){\left[\left[#1,#2\right],#3\right] + \left[\left[#3,#1\right],#2\right] + \left[\left[#2,#3\right],#1\right]}

\DeclareMathOperator{\Div}{div} 

\DeclareMathOperator{\trace}{trace}

\numberwithin{equation}{section}
\allowdisplaybreaks

\def\jid(#1#2#3){\left[\left[#1,#2\right],#3\right] + \left[\left[#3,#1\right],#2\right] + \left[\left[#2,#3\right],#1\right]} 
\def\BV(#1#2){B^{\V}(#1,#2)}
\def\hproj(#1#2){\frac{1}{2}\,\H\,(\nab(#1#2) + \nab(#2#1))}
\def\BH(#1#2){B^{\H}(#1,#2)}
\def\vproj(#1#2){\frac{1}{2}\,\V\,(\nab(#1#2) + \nab(#2#1))}

\def\lieb(#1#2){\left[#1,#2\right]}

\begin{document}

\title
[Proper Biharmonic Maps and $(2,1)$-Harmonic Morphisms]
{Proper Biharmonic Maps and $(2,1)$-Harmonic Morphisms from Some Wild Geometries}

\author{Elsa Ghandour}
\address{Mathematics, Faculty of Science\\
University of Lund\\
Box 118, Lund 221\\
Sweden}
\email{Elsa.Ghandour@math.lu.se}

\author{Sigmundur Gudmundsson}
\address{Mathematics, Faculty of Science\\
	University of Lund\\
	Box 118, Lund 221\\
	Sweden}
\email{Sigmundur.Gudmundsson@math.lu.se}

\begin{abstract}
In this work we construct a variety of new complex-valued proper biharmonic maps and $(2,1)$-harmonic morphisms on Riemannian manifolds with non-trivial geometry.  These are solutions to a non-linear system of partial differential equations depending on the geometric data of the manifolds involved.
\end{abstract}



\dedicatory{COI statement: All authors declare that they have no conflicts of interest}

\maketitle

\section{Introduction}

The concept of a {\it harmonic morphism} $\phi:(M,g)\to(N,h)$, between Riemannian manifolds, was introduced by Fuglede and Ishihara in the late 1970s independently, see \cite{Fug} and \cite{Ish}. These are maps pulling back local real-valued harmonic functions on $N$ to harmonic functions on $M$.  These objects have an interesting connection with the geometry of the manifolds involved and have lead to vibrant research activities, as can be traced in the excellent work \cite{Bai-Woo-book}, by Baird and Wood, and the regularly updated online bibliography \cite{Gud-bib}, maintained by the second author.

Recently, the notion was generalised to  $(p,q)$-{\it harmonic morphisms}, pulling back real-valued $q$-harmonic functions on $N$ to $p$-harmonic functions on $M$, see \cite{Gha-Gud-1}.  The case of $(2,1)$ had earlier been studied in \cite{Gha-Ou-1} under the name {\it generalised harmonic morphisms}.  In \cite{Gha-Gud-1}, the authors characterise complex-valued $(p,q)$-harmonic morphisms  $\phi:(M,g)\to\cn$ in terms of a heavily non-linear system of partial differential equations.  They also provide methods for producing explicit solutions in the case when the domain $(M,g)$ is the  $m$-dimensional Euclidean space.

The principal aim of this work is to extend the study to complex-valued $(2,1)$-harmonic morphisms from Riemannian manifolds $(M,g)$.  We model our manifolds $M$ as open subsets of $\rn^m$, equipped with a Riemannian metric $g$ of a particular form, see Section \ref{section-geometries}.  We then investigate when the natural projection $\Phi:(M,g)\to\cn$ onto the first two coordinates is {\it horizontally conformal}, {\it harmonic} and even {\it biharmonic}.  This leads to a {\it non-linear} system of partial differential equations involving the geometric data on $(M,g)$.  We then find several explicit solutions and thereby construct metrics $g$ turning the projection $\Phi$ into a proper biharmonic map and even a proper $(2,1)$-harmonic morphism.  By this we construct the first known complex-valued $(2,1)$-harmonic morphisms from Riemannian manifolds with non-trivial geometry.  Since the problem is invariant under conformal changes of the metric on $(N,h)$ this provides local solutions to any Riemann surface, see Proposition  \ref{proposition-holomorphic-composition}.

\section{Preliminaries}\label{section-preliminaries}

Let $(M,g)$ be an $m$-dimensional  Riemannian manifold and $T^{\cn}M$ be the complexification of the tangent bundle $TM$ of $M$. We extend the metric $g$ to a complex-bilinear form on $T^{\cn}M$.  Then the gradient $\nabla \phi$ of a complex-valued function $\phi:(M,g)\to\cn$ is a section of $T^{\cn}M$.  In this situation, the well-known complex linear {\it Laplace-Beltrami operator} (alt. {\it tension field}) $\tau$ on $(M,g)$ acts locally on $\phi$ as follows
$$
\tau(\phi)=\Div (\nabla \phi)=\sum_{i,j=1}^m\frac{1}{\sqrt{|g|}} \frac{\partial}{\partial x_j}
\left(g^{ij}\, \sqrt{|g|}\, \frac{\partial \phi}{\partial x_i}\right).
$$
For two complex-valued functions $\phi,\psi:(M,g)\to\cn$ we have the following well-known relation
\begin{equation}\label{equation-basic}
	\tau(\phi\cdot \psi)=\tau(\phi)\cdot \psi +2\cdot\kappa(\phi,\psi)+\phi\cdot\tau(\psi),
\end{equation}
where the complex bilinear {\it conformality operator} $\kappa$ is given by $\kappa(\phi,\psi)=g(\nabla \phi,\nabla \psi)$.  Locally this satisfies 
$$\kappa(\phi,\psi)=\sum_{i,j=1}^mg^{ij}\cdot\frac{\partial \phi}{\partial x_i}\frac{\partial \psi}{\partial x_j}.$$

We are now ready to define the complex-valued proper $p$-harmonic functions.

\begin{definition}\label{definition-proper-r-harmonic}
For a positive integer $p$, the iterated Laplace-Beltrami operator $\tau^p$ is given by
$$\tau^{0} (\phi)=\phi\ \ \text{and}\ \ \tau^p (\phi)=\tau(\tau^{(p-1)}(\phi)).$$	We say that a complex-valued function $\phi:(M,g)\to\cn$ is
\begin{enumerate}
\item[(a)] {\it $p$-harmonic} if $\tau^p (\phi)=0$, and
\item[(b)] {\it proper $p$-harmonic} if $\tau^p (\phi)=0$ and $\tau^{(p-1)} (\phi)$ does not vanish identically.
\end{enumerate}
\end{definition}

We now introduce the natural notion of a $(p,q)$-harmonic morphism.  For $(p,q)=(1,1)$ this is the classical case of harmonic morphisms introduced by Fuglede and Ishihara, in \cite{Fug} and \cite{Ish}, independently.

\begin{definition}\label{definition-pq}
	A map $\phi:(M,g)\to (N,h)$ between Riemannian manifolds is said to be a {\it $(p,q)$-harmonic morphism} if, for any {\it $q$-harmonic} function $f:U\subset N\to\rn$, defined on an open subset $U$ such that $\phi^{-1}(U)$ is not empty, the composition $f\circ\phi:\phi^{-1}(U)\subset M\to\rn$ is  {\it $p$-harmonic}.
\end{definition}

As an immediate consequence of Definition \ref{definition-pq} we have the following natural composition law.

\begin{lemma}\label{lemma-composition}
	Let $\phi:(M,g)\to (\bar N,\bar h)$ be a $(p,r)$-harmonic morphism between Riemannian manifolds.  If $\psi:(\bar N,\bar h)\to (N,h)$ is an $(r,q)$-harmonic morphism then  the composition $\psi\circ\phi:(M,g)\to (N,h)$ is a $(p,q)$-harmonic morphism.
\end{lemma}

\section{Some Rather Wild Geometries}
\label{section-geometries}

The aim of this section is to describe a particular collection of Riemannian manifolds $(M,g)$ investigated in this work. As far as we know they have not been studied before in the geometric literature.  We present formulae for their sectional curvatures to show that the geometry is here far from being trivial.  We then give two concrete examples that turn out to be useful later on.
\smallskip

For an open subset $M$ of $\rn^m$, let $\lambda,\lambda_1,\dots,\lambda_m:M\to\rn$ be $C^3$-functions such that $\lambda=\lambda_1=\lambda_2$ and equip the manifold $M$ with the Riemannian metric $g$ of the special form
$$g=e^{-2\lambda}(dx^2+dy^2)+e^{-2\lambda_3}dx_3^2+\cdots +e^{-2\lambda_m}dx_m^2.$$ 
For our purposes it is practical to introduce the function $f:M\to\rn$ with $$f({\bf x})=\sum_{k=3}^m\lambda_k({\bf x}).$$
For the tangent bundle $TM$ of $(M,g)$ we have the following global orthonormal frame 
$$
\left\{
X_1=e^{\lambda_1}\cdot\frac{\partial}{\partial x_1},\,
\dots\, ,
X_m=e^{\lambda_m}\cdot\frac{\partial}{\partial x_m}
\right\}$$
When appropriate, we shall by ${\bf x}=(x,y,x_3,\dots,x_m)$ denote the canonical coordinates $(x_1,\dots,x_m)$ on $\rn^m$  and set $X=X_1$, $Y=X_2$.
The Lie brackets for $TM$ satisfy
$$\lb {X_j}{X_k}=
e^{\lambda_j}\,(\lambda_k)_{x_j}X_k
-e^{\lambda_k}\,(\lambda_j)_{x_k}X_j, $$
where the subscript $x_j$ means the partial derivative with respect to the j-th coordinate function.  A standard computation shows that for  the sectional curvature $K(X_j\wedge X_k)$ of the $2$-plane $X_j\wedge X_k$ we have 
\begin{eqnarray*}
K(X_j\wedge X_k)
&=&
e^{2\lambda_j}[(\lambda_j)_{x_j}(\lambda_k)_{x_j}
+(\lambda_k)_{x_jx_j}-(\lambda_k)_{x_j}^2]\\
& &
+\,e^{2\lambda_k}[(\lambda_k)_{x_k}(\lambda_j)_{x_k}
+(\lambda_j)_{x_kx_k}-(\lambda_j)_{x_k}^2]\\
& &
-\sum_{r\notin\{j,k\}}
e^{2\lambda_r}(\lambda_k)_{x_r}(\lambda_j)_{x_r}.
\end{eqnarray*}
In particular, for the horizonal section $X\wedge Y$ we have 
\begin{eqnarray*}
K(X\wedge Y)
&=&
e^{2\lambda}(\lambda_{xx}+\lambda_{yy})
-\sum_{k=3}^me^{2\lambda_k}\lambda_{x_k}^2.
\end{eqnarray*}

Let $\Phi:(\rn^m,g)\to\cn$ be the horizontally conformal submersion $$\Phi:{\bf x}\mapsto (x+iy)\cong (x\cdot\text{\bf e}_1+y\cdot\text{\bf e}_2)$$ 
with dilation $e^\lambda:\rn^m\to\rn^+$.
For the tangent bundle $TM$ we have the following  orthogonal decomposition $TM=\H\oplus\V$ into its horizontal and vertical subbundles $\H$ and $\V$, respectively, where 
$$\H=\text{span}\{X,Y\}\ \ \text{and}\ \ \V=\text{span}\{X_3,\dots ,X_m\}.$$

\begin{definition}
For an open subset $M$ of $\rn^m$ we denote by $\Omega(M)$ the set of $C^3$-functions $\omega:M\to\rn$ which are independent of the first coordinates $(x,y)$ of ${\bf x}=(x,y,x_3,\dots,x_m)$ i.e. 
$$\Omega(M)=\{\omega\in C^3(M,\rn)|\ \omega_x=\omega_y=0\}.$$
\end{definition}

We now present two Riemannian manifolds $(M,g)$ with non-trivial geometry.  Later in this work, we then show that the complex-valued function $\Phi:(M,g)\to\cn$ is proper biharmonic in these and other similar cases.

\begin{example}
For an open subset $M$ of $\rn^3$, constants  $A,B,\theta\in\rn$ and $\alpha\in\Omega(M)$ let the functions $\lambda,f:M\to\rn$ be defined by 
$$
\lambda (x,y,z)=\alpha(z),\ \ 
f(x,y,z)=\log(1 + \tan(\Theta)^2),
$$
where 
$$\Theta(x,y)=A\cdot (\cos\theta\cdot x+\sin\theta\cdot y)+B.$$
Then equip $M$ with the Riemannian metric $g$ given by 
$$g=e^{-2\lambda}(dx^2+dy^2)+e^{-2f}dz^2.$$
Then the sectional curvature function $K$ of the manifold $(M^3,g)$ satisfies
$$
K(X\wedge Y)=
-\frac
{\lambda_z^2}
{\cos^4(\Theta)},
$$
$$
K(X\wedge Z)=
\frac
{\lambda_{zz}-\lambda_z^2+2A^2e^{2\lambda}\cos^2(\theta)\cdot (2\cos^4(\Theta)-\cos^2(\Theta))}
{\cos^4(\Theta)},
$$
$$
K(Y\wedge Z)=
\frac
{\lambda_{zz}-\lambda_z^2+2A^2e^{2\lambda}\sin^2(\theta)\cdot (2\cos^4(\Theta)-\cos^2(\Theta))}
{\cos^4(\Theta)}.
$$

If we assume that $A,B,\theta\in\Omega(M)$, rather than  $A,B,\theta\in\rn$, then the geometry of $(M,g)$ runs rather wild.  The formulae for the sectional curvature $K$ become far too extensive to be included in this work. For explicit proper biharmonic maps in that general case, see Example \ref{example-mother-1}.
\end{example}

\begin{example}
For an open subset $M$ of $\rn^4$, constants $A,B,\Psi\in\rn$ and $\alpha\in\Omega(M)$, let the functions $\lambda,f:M\to\rn$ be defined by 
$$
\lambda ({\bf x})=\alpha(z,w),$$
$$f({\bf x})=\lambda_3({\bf x})+\lambda_4({\bf x})= -2\log(A\cdot(\cos(t)\cdot x +\sin(t)\cdot y)+ B),$$
where
$$\lambda_3({\bf x})=-\log(A\cdot(\cos(t)\cdot x +\sin(t)\cdot y)+ B)+\Psi,$$
$$\lambda_4({\bf x})=-\log(A\cdot(\cos(t)\cdot x +\sin(t)\cdot y)+ B)-\Psi.$$ 
Then equip $M$ with the Riemannian metric $g$ satifying 
$$g=e^{-2\lambda}(dx^2+dy^2)+e^{-2\lambda_3}dz^2+e^{-2\lambda_4}dw^2.$$
Then a standard computation shows that the sectional curvatures of $(M,g)$ fulfill 
$$K(X\wedge Y)=\frac
{e^{2\Psi}\cdot \lambda^2_{z}+e^{-2\Psi}\cdot\lambda^2_{w}}
{(A\cdot (\cos(t)\cdot x +\sin(t)\cdot y)+B)^2},
$$
$$
K(X\wedge Z)=K(Y\wedge Z)=\frac
{e^{2\Psi}\cdot (\lambda_{zz}-\lambda_z^2)}{(A\cdot (\cos(t)\cdot x +\sin(t)\cdot y)+B)^2},$$
$$
K(X\wedge W)=K(Y\wedge W)=\frac
{e^{-2\Psi}\cdot (\lambda_{ww}-\lambda_w^2)}{(A\cdot (\cos(t)\cdot x +\sin(t)\cdot y)+B)^2},
$$
$$
K(Z\wedge W)= -\frac{e^{2\lambda}\cdot A^2}{(A\cdot (\cos(t)\cdot x +\sin(t)\cdot y) + B)^2}.
$$
\smallskip

\noindent	
If we assume that $A,B,\Psi\in\Omega(M)$, rather than  $A,B,\Psi\in\rn$, then the formulae for the sectional curvature $K$ become very complicated, including partial derivatives of these functions. For explicit proper $(2,1)$-harmonic morphisms in that general case, see Example \ref{example-(2,1)-1}

\end{example}

\section{The tension Fields $\tau(\Phi)$ and $\tau^2(\Phi)$}
\label{section-tension-fields}

Our first principal aim is to construct Riemannian manifolds $(M,g)$, of the form introduced in Section \ref{section-geometries}, such that the  horizontally conformal submersion $\Phi:(M,g)\to\cn$ with   $$\Phi:{\bf x}\mapsto (x+iy)\cong (x\cdot\text{\bf e}_1+y\cdot\text{\bf e}_2)$$ 
is a proper biharmonic map.  For this purpose we now want to determine the tension field $\tau(\Phi)$ and the bitension field $\tau^2(\Phi)$ of $\Phi$, respectively.

\begin{lemma}\label{lemma-sum-DXX}
Let $(M^m,g)$ be a Riemannian manifold, as defined above, with the orthonormal basis $\{X_1,\dots ,X_m\}$ for the tangent bundle $TM$.  Then its Levi-Civita connection satisfies
$$
\sum_{k=1}^m\nabla_{X_k}{X_k}
=\sum_{j=1}^m\sum_{k\neq j}^me^{\lambda_j}\,(\lambda_k)_{x_j}X_j.
$$
\end{lemma}

\begin{proof}
The statement follows from the following computation
\begin{eqnarray*}
\sum_{k=1}^m\nabla_{X_k}{X_k}
&=&
\sum_{j,k=1}^mg(\nabla_{X_k}{X_k},X_j)\,X_j\\
&=&
\sum_{j,k=1}^mg(\lb{X_j}{X_k},X_k)\,X_j\\
&=&
\sum_{j,k=1}^mg( 
e^{\lambda_j}\,(\lambda_k)_{x_j}X_k
-e^{\lambda_k}\,(\lambda_j)_{x_k}X_j,X_k)\,X_j\\
&=&
 \sum_{j,k=1}^me^{\lambda_j}\,(\lambda_k)_{x_j}\,X_j
-  \sum_{j=1}^me^{\lambda_j}\,(\lambda_j)_{x_j}\,X_j\\
&=&
\sum_{j=1}^m\sum_{k\neq j}^m
e^{\lambda_j}\,(\lambda_k)_{x_j}X_j.
\end{eqnarray*}
\end{proof}

\begin{lemma}\label{lemma-sum-dphi(DXX)}
Let $\Phi:(M^m,g)\to\cn$ be the horizontally conformal submersion $$\Phi:{\bf x}\mapsto (x+iy)\cong (x\cdot\text{\bf e}_1+y\cdot\text{\bf e}_2)$$ 
with dilation $e^\lambda:M\to\rn^+$. Then we have the following relation
$$
\sum_{k=1}^md\Phi(\nabla_{X_k}{X_k})
=e^{2\lambda}\,(\lambda_{x}+f_{x})\cdot\text{\bf e}_1
+e^{2\lambda}\,(\lambda_{y}+f_{y})\cdot\text{\bf e}_2.
$$
\end{lemma}

\begin{proof}
It follows from Lemma \ref{lemma-sum-DXX} and the fact that the differential $d\Phi$ satisfies $d\Phi(X_3)=\cdots =d\Phi(X_m)=0$ that  
\begin{eqnarray*}
& &\sum_{k=1}^md\Phi(\nabla_{X_k}{X_k})\\
&=&\sum_{j=1}^m\sum_{k\neq j}^me^{\lambda_j}\,
(\lambda_k)_{x_j}d\Phi(X_j)\\
&=&
 e^{\lambda_1}\,(\lambda_2)_{x_1}d\Phi(X_1)
+e^{\lambda_2}\,(\lambda_1)_{x_2}d\Phi(X_2)
+\sum_{j=1}^2\sum_{k=3}^me^{\lambda_j}\,
(\lambda_k)_{x_j}d\Phi(X_j)\\
&=&
 e^{2\lambda}\,\lambda_{x_1}\text{\bf e}_1
+e^{2\lambda}\,\lambda_{x_2}\text{\bf e}_2
+\sum_{j=1}^2e^{\lambda_j}\,
f_{x_j}d\Phi(X_j)\\
&=&
 e^{2\lambda}\,(\lambda_{x}+f_{x})\,\text{\bf e}_1
+e^{2\lambda}\,(\lambda_{y}+f_{y})\,\text{\bf e}_2.
\end{eqnarray*}
\end{proof}

With the next result we provide a formula for the tension field $\tau(\Phi)$ of the horizontally conformal submersion $\Phi$.

\begin{proposition}\label{proposition-tension-field}
Let $\Phi:(M,g)\to\cn$ be the horizontally conformal submersion $$\Phi:{\bf x}\mapsto (x+iy)\cong (x\cdot\text{\bf e}_1+y\cdot\text{\bf e}_2)$$ 
with dilation $e^\lambda:M\to\rn^+$.  Then the tension field $\tau(\Phi)$ of $\Phi$ satisfies
\begin{eqnarray*}
\tau(\Phi)&=&
-\,e^{2\lambda}\,(
 f_{x}\cdot\text{\bf e}_1
+f_{y}\cdot\text{\bf e}_2).
\end{eqnarray*}
\end{proposition}

\begin{proof}
The two vector fields $X_1$ and $X_2$ generate the horizontal distribution $\H$ so the horizontal conformality of $\Phi$ is a direct consequence of the fact that 
$$
d\Phi(X_1)=e^{\lambda}\cdot\text{\bf e}_1,\  
d\Phi(X_2)=e^{\lambda}\cdot\text{\bf e}_2, \ 
d\Phi(X_3)=0,\ \dots\ ,d\Phi(X_m)=0.
$$
The tension field $\tau(\Phi)$ of $\Phi$ is defined by the well-known formula 
\begin{eqnarray*}
\tau(\Phi)&=& \sum_{k=1}^m\left\{\nabla_{X_k}^\Phi d\Phi(X_k)-d\Phi(\nabla_{X_k}X_k)\right\}.
\end{eqnarray*}
For the first part, we have 
\begin{eqnarray*}
\sum_{k=1}^m\nabla_{X_k}^\Phi d\Phi(X_k)
&=&
\nabla^\Phi_{X_1}d\Phi(X_1)+\nabla^\Phi_{X_2}d\Phi(X_2)\\
&=&
e^{\lambda}\cdot\frac{\partial}{\partial_{x_1}}(e^{\lambda}\,\text{\bf e}_1)+
e^{\lambda}\cdot\frac{\partial}{\partial_{x_2}}(e^{\lambda}\,\text{\bf e}_2)\\
&=&e^{2\lambda}\,(\lambda_{x}\cdot\text{\bf e}_1
 +\lambda_{y}\cdot\text{\bf e}_2).
\end{eqnarray*}
For the second part, we now employ Lemma \ref{lemma-sum-DXX} and yield 
$$
\sum_{k=1}^md\Phi(\nabla_{X_k}{X_k})
=e^{2\lambda}\,(\lambda_{x}+f_{x})\cdot\text{\bf e}_1
+e^{2\lambda}\,(\lambda_{y}+f_{y})\cdot\text{\bf e}_2.
$$
The statement is now an immediate consequence of the above calculations.
\end{proof}

\begin{corollary}
The horizontally conformal submersion $\Phi:(M,g)\to\cn$, with 
$$\Phi:{\bf x}\mapsto (x+iy)\cong (x\cdot\text{\bf e}_1+y\cdot\text{\bf e}_2),$$ is harmonic and hence a harmonic morphism if and only if $(f_{x},f_{y})=0$.
\end{corollary}

\begin{proof}
This is an immediate consequence of Proposition \ref{proposition-tension-field} and the characterisation of harmonic morphisms, proven by Fuglede and Ishihara in \cite{Fug} and \cite{Ish}, respectively.
\end{proof}
 
After determining the tension field $\tau(\Phi)$, we now turn our attention to the bitension field $\tau^2(\Phi)$.

\begin{definition}\label{definition-diff-operators}
For an open subset $M$ of $\rn^m$, let $\lambda,f:M\to\rn$ be differentiable functions on $M$ with coordinates ${\bf x}=(x,y,x_3,\dots,x_m)$. Then we define the {\it non-linear} partial differential operators $D_1,D_2$ by 
\begin{eqnarray*}
	D_1(\lambda,f)
	&=&
	\{\,(\lambda_{x}+f_{x})\cdot
	(2\lambda_{x}f_{x}+f_{xx})+\,(\lambda_{y}+f_{y})\cdot
	(2\lambda_{y}f_{x}+f_{xy})\\
	& &
	\qquad\qquad -(6\,\lambda_{x}^2\,f_{x}+2\,\lambda_{xx}\,f_{x}
	+5\,\lambda_{x}\,f_{xx}+f_{xxx})\\
	& &
	\qquad\qquad -(6\,\lambda_{y}^2\,f_{x}+2\,\lambda_{yy}\,f_{x}
	+5\,\lambda_{y}\,f_{xy}+f_{xyy})\,\},
\end{eqnarray*}
\begin{eqnarray*}
	D_2(\lambda,f)
	&=&
	\{\,(\lambda_{x}+f_{x})\cdot 
	(2\lambda_{x}f_{y}+f_{yx})
	+\,(\lambda_{y}+f_{y})\cdot (2\lambda_{y}f_{y}+f_{yy})\\
	& &
	\qquad\qquad -(6\,\lambda_{x}^2f_{y}+2\,\lambda_{xx}f_{y}
	+5\,\lambda_{x}\,f_{yx}+f_{yxx})\\
	& &
	\qquad\qquad -(6\,\lambda_{y}^2f_{y}+2\,\lambda_{yy}f_{y}
	+5\,\lambda_{y}\,f_{yy}+f_{yyy})\,\}.
\end{eqnarray*}
\end{definition}

With the next result we present a formula for the bitension field $\tau^2(\Phi)$ of the horizontally conformal submersion $\Phi$.
 
\begin{theorem}
Let $\Phi:(M,g)\to\cn$ be the horizontally conformal submersion $$\Phi:{\bf x}\mapsto (x+iy)\cong (x\cdot\text{\bf e}_1+y\cdot\text{\bf e}_2)$$ 
with dilation $e^\lambda:M\to\rn^+$. Then the bitension field $\tau^2(\Phi)$ of $\Phi$ satisfies 
\begin{eqnarray*}
	\tau^2(\Phi)&=&
	e^{4\lambda}\cdot D_1(\lambda,f)\cdot\text{\bf e}_1
	+e^{4\lambda}\cdot D_2(\lambda,f)\cdot\text{\bf e}_2.
\end{eqnarray*}
\end{theorem}
\smallskip
 
 \begin{proof}
The bitension field $\tau^2(\Phi)$ of the $C^4$-map $\Phi$ is given by 
\begin{eqnarray*}
\tau^2(\Phi)
&=&
\sum_{k=1}^m\{
\nabla_{X_k}^\Phi\nabla_{X_k}^\Phi\tau(\Phi)
-\nabla^\Phi_{\nabla_{X_k}X_k}\tau(\Phi)\}.
\end{eqnarray*}
First we notice that for $k=1,2$ we have 
\begin{eqnarray*}
\nabla_{X_k}^\Phi\tau(\phi)
&=&
-\,e^{\lambda}\cdot\frac{\partial}{\partial x_k}\big( e^{2\lambda}\,\{
f_{x_1}\,\text{\bf e}_1+f_{x_2}\,\text{\bf e}_2\}\big)\\
&=&
-\,e^{3\lambda}\{
(2\,\lambda_{x_k}f_{x_1}+f_{x_1x_k})\,\text{\bf e}_1\\
& &
\qquad +\,(2\,\lambda_{x_k}f_{x_2}+f_{x_2x_k})\,\text{\bf e}_2\}
\end{eqnarray*}
Differentiating once more gives 
\begin{eqnarray*}
& &\sum_{k=1}^m\nabla_{X_k}^\Phi\nabla_{X_k}^\Phi\tau(\phi)\\
&=&
-\,\sum_{k=1}^2e^{\lambda}\cdot\frac{\partial}{\partial x_k}
(e^{3\lambda}(2\lambda_{x_k}f_{x_1}+f_{x_1x_k}))\,\text{\bf e}_1\\
& &
-\,\sum_{k=1}^2e^{\lambda}\cdot\frac{\partial}{\partial x_k}
(e^{3\lambda}(2\,\lambda_{x_k}f_{x_2}+f_{x_2x_k}))\,\text{\bf e}_2\\
&=&-\,e^{4\lambda}\cdot\sum_{k=1}^2
(6\,\lambda_{x_k}^2f_{x_1}+2\,\lambda_{x_kx_k}f_{x_1}
+5\,\lambda_{x_k}\,f_{x_1x_k} +f_{x_1x_kx_k})\,\text{\bf e}_1\\
& &-\,e^{4\lambda}\cdot\sum_{k=1}^2
(6\,\lambda_{x_k}^2f_{x_2}+2\,\lambda_{x_kx_k}f_{x_2}
+5\,\lambda_{x_k}\,f_{x_2x_k} +f_{x_2x_kx_k})\,\text{\bf e}_2.
 \end{eqnarray*}
For the second part of the bitension field $\tau^2(\Phi)$ of $\Phi$ we now yield
\begin{eqnarray*}
& &-\sum_{k=1}^m\nabla^\Phi_{\nabla_{X_k}X_k}(\tau(\Phi))\\
&=&-e^{2\lambda}\,(\lambda_{x_1}+f_{x_1})\cdot\frac{\partial}{\partial x_1}\,\tau(\Phi)
-e^{2\lambda}\,(\lambda_{x_2}+f_{x_2})\cdot\frac{\partial}{\partial x_2}\,\tau(\Phi)\\
&=&e^{2\lambda}\,(\lambda_{x_1}+f_{x_1})\cdot\frac{\partial}{\partial x_1}(\,e^{2\lambda}\cdot\{f_{x_1}\,\text{\bf e}_1
+f_{x_2}\,\text{\bf e}_2\})\\
& &
+\,e^{2\lambda}\,(\lambda_{x_2}+f_{x_2})\cdot\frac{\partial}{\partial x_2}(e^{2\lambda}\cdot\{f_{x_1}\,\text{\bf e}_1
+f_{x_2}\,\text{\bf e}_2\})\\
&=&e^{4\lambda}\,(\lambda_{x_1}+f_{x_1})\cdot
\{(2\lambda_{x_1}f_{x_1}+f_{x_1x_1})\,\text{\bf e}_1
+(2\lambda_{x_1}f_{x_2}+f_{x_2x_1})\,\text{\bf e}_2\}\\
& &
+\,e^{4\lambda}\,(\lambda_{x_2}+f_{x_2})\cdot
\{(2\lambda_{x_2}f_{x_1}+f_{x_1x_2})\,\text{\bf e}_1
+(2\lambda_{x_2}f_{x_2}+f_{x_2x_2})\,\text{\bf e}_2\}
\end{eqnarray*}
We now easily obtain the stated result by adding the terms.
\end{proof}

\section{Explicit Proper Biharmonic Submersions.}
\label{section-explicit-examples}

In the Section \ref{section-tension-fields} we have derived explicit formulae for the tension fields $\tau (\Phi)$ and $\tau^2(\Phi)$.  This leads to a system of {\it non-linear} partial differential equations for the pair of functions $(\lambda,f)$. We are now interested in constructing Riemannian metrics $g$, on open subsets $M$ of $\rn^m$, turning the horizontally conformal submersion $\Phi:(M,g)\to\cn$ into {\it proper} biharmonic maps i.e. finding explicit solutions $(\lambda,f)$ to the system
$$(f_x,f_y)\neq 0,\ \ D_1(\lambda,f)=0\ \ \text{and}\ \ D_2(\lambda,f)=0.$$ 
Let us first consider the case when $\lambda=\alpha\in\Omega(M)$ i.e.  independent of the two first coordinates $x$ and $y$. Then the differential operators $D_1$ and $D_2$ simplify to 
\begin{eqnarray*}
D_1(\alpha,f)
&=&
(f_{x}\cdot f_{xx}+f_{y}\cdot f_{xy}-f_{xxx}-f_{xyy})\\
&=&
\tfrac 12(f_x^2+f_y^2-2\,(f_{xx}+f_{yy}))_x,
\end{eqnarray*}
\begin{eqnarray*}
D_2(\alpha,f)
&=&
(f_{x}\cdot f_{yx}+f_{y}\cdot f_{yy}-f_{yxx}-f_{yyy})\\
&=&
\tfrac 12(f_x^2+f_y^2-2\,(f_{xx}+f_{yy}))_y.
\end{eqnarray*}
As a first example we have the following which clearly gives solutions to the system under consideration.

\begin{example}\label{example-biharmonic-1}
For an open subset $M$ of $\rn^m$ and $A,B,\alpha,\beta\in\Omega(M)$ let the functions $\lambda,f:M\to\rn$ be defined by 
$$\lambda ({\bf x})=\alpha,\ \ f({\bf x})=A\cdot x+B\cdot y +\beta.$$ 
If $A^2+B^2\neq 0$, then the associated map $\Phi:(M,g)\to\cn$ is  horizontally conformal and proper biharmonic i.e.
$$(f_x,f_y)\neq 0, \ \ D_1(\alpha,f)=0
\ \ \text{and}\ \  D_2(\alpha,f)=0.$$
It is clear that the choice of $M$ and $A,B,\alpha,\beta\in\Omega(M)$ can lead to rather non-trivial geometries $(M,g)$.
\end{example}

If we now assume that the function $f:M\to\rn$ is independent of the coordinate $y$ then we have that 
$$f_x\neq 0,\ \ D_1(\alpha,f)
=f_x\cdot f_{xx}-f_{xxx}=0\ \ \text{and}\ \ D_2(\alpha,f)=0.$$
The ordinary differential equation 
$$D_1(\alpha,f)=f_x\cdot f_{xx}-f_{xxx}=\frac 12(f_x^2-2f_{xx})_x=0$$
can easily be integrated to
$$f_x({\bf x})=A\cdot\tan(\frac{A\cdot x}2+B),$$
for some $A,B\in\Omega(M)$. Integrating yet again, we finally obtain 
$$f({\bf x})=\log (1 + \tan(\frac{A\cdot x}2+B)^2)+\beta,$$
defined on the appropriate open subset $M$ of $\rn^m$ and with $\beta\in\Omega(M)$.  From this we see that under the above mentioned assumptions and modulo the functions $A,B,\beta\in\Omega(M)$, the solution is uniquely determined. This leads to the following.

\begin{example}\label{example-mother-1}
For an open subset $M$ of $\rn^m$ and $A,B,\alpha,\beta,\theta\in\Omega(M)$ let the functions $\lambda,f:M\to\rn$ be defined by 
$$\lambda ({\bf x})=\alpha,\ \ f({\bf x})=
\log(1 + \tan({A\cdot (\cos\theta\cdot x+\sin\theta\cdot y)} + B)^2)+\beta.$$
If $A\neq 0$, then the associated horizontally conformal map $\Phi:M\to\cn$ is proper biharmonic.	
\end{example}

By the seperation of variables, one easily yields the next two families of solutions.

\begin{example}
For an open subset $M$ of $\rn^m$ and  $A,B,C,D,\alpha,\beta\in\Omega(M)$ let the functions $\lambda,f:M\to\rn$ be defined by 
$$\lambda ({\bf x})=\alpha,\ \ f({\bf x})=
\log((1 + \tan({A\cdot x} + C)^2)\cdot (1 + \tan({B\cdot y} + D)^2))+\beta.$$
If $A^2+B^2\neq 0$ then the associated horizontally conformal map $\Phi:(M,g)\to\cn$ is proper biharmonic.
\end{example}

\begin{example}
For an open subset $M$ of $\rn^m$ and  $A,B,C,D,\alpha,\beta\in\Omega(M)$ let the functions $\lambda,f:M\to\rn$ be defined by 
$$\lambda ({\bf x})=\alpha,\ \ f({\bf x})=
-2\,\log((A\cdot x + C)(B\cdot y +D))+\beta.$$
If $A^2+B^2\neq 0$ then the associated horizontally conformal map $\Phi:(M,g)\to\cn$ is proper biharmonic.
\end{example}

We also yield the following examples without assuming the condition  $\lambda\in\Omega(M)$.

\begin{example}
For an open subset $M$ of $\rn^m$ and  $A,B,C,D,\alpha,\beta\in\Omega(M)$ let the functions $\lambda,f:M\to\rn$ be defined by 
$$\lambda ({\bf x})=A\cdot x+B\cdot y+\alpha,\ \ 
f({\bf x})=C\cdot x +D\cdot y+\beta.$$
If $A\,C+B\,D=2\,(A^2+B^2)\neq 0$ then the associated horizontally conformal map $\Phi:(M,g)\to\cn$ is proper biharmonic.
\end{example}

\begin{example}
For an open subset $M$ of $\rn^m$ and  $A,B,r,\alpha,\beta,\theta\in\Omega(M)$ let the functions $\lambda,f:M\to\rn$ be defined by
$$\lambda({\bf x})=-r\cdot (\cos\theta\cdot x+\sin\theta\cdot y)+\alpha,$$
$$f({\bf x})= A\cdot \exp(2\,r\cdot(\cos\theta\cdot x + \sin\theta\cdot y)+B) + \beta.$$
If $r\neq 0$, then the associated horizontally conformal map $\Phi:(M,g)\to\cn$ is proper biharmonic.
\end{example}

\begin{example}	For an open subset $M$ of $\rn^m$ and  $A,B,r,\alpha,\beta,\theta\in\Omega(M)$ let the functions $\lambda,f:M\to\rn$ be defined by
$$\lambda({\bf x})=\tfrac 12\cdot\log(A\cdot\exp(r\cdot(\cos\theta\cdot x+\sin\theta\cdot y))+B)+\alpha,$$
$$f({\bf x})= r\cdot(\cos\theta\cdot x + \sin\theta\cdot y) + \beta.$$
	If $r\neq 0$, then the associated horizontally conformal map $\Phi:(M,g)\to\cn$ is proper biharmonic.
\end{example}

\begin{example}	For an open subset $M$ of $\rn^m$ and  $A,\alpha,\beta,\theta\in\Omega(M)$ let the functions $\lambda,f:M\to\rn$ be defined by
$$\lambda({\bf x})=\tfrac 12\cdot\log(\cos\theta\cdot x+\sin\theta\cdot y)+\alpha,$$
$$f({\bf x})= A\cdot\log(\cos\theta\cdot x + \sin\theta\cdot y) + \beta.$$
Then the associated horizontally conformal map $\Phi:(M,g)\to\cn$ is proper biharmonic.
\end{example}

\begin{example}\label{example-biharmonic-9}
For an open subset $M$ of $\rn^m$ and  $A,\alpha,\beta,\theta\in\Omega(M)$ let the functions $\lambda,f:M\to\rn$ be defined by
$$\lambda({\bf x})=A\cdot\log(\cos\theta\cdot x+\sin\theta\cdot y)+\alpha,$$
$$f({\bf x})= 2\,(A-1)\cdot\log(\cos\theta\cdot x + \sin\theta\cdot y) + \beta.$$
Then the associated horizontally conformal map $\Phi:(M,g)\to\cn$ is proper biharmonic if and only if $A\neq 1$.
\end{example}

\section{The tension Fields $\tau(\Phi^2)$ and $\tau^2(\Phi^2)$}
\label{section-tension-fields-2}

Our second principal aim is to construct Riemannian manifolds $(M,g)$, of the form introduced in Section \ref{section-geometries}, such that the  horizontally conformal submersion $\Phi:(M,g)\to\cn$ with   $$\Phi:{\bf x}\mapsto (x+iy)\cong (x\cdot\text{\bf e}_1+y\cdot\text{\bf e}_2)$$ 
is a proper $(2,1)$-harmonic morphism.  For this purpose we now want to determine the tension field $\tau(\Phi^2)$ and the bitension field $\tau^2(\Phi^2)$ of $\Phi^2$, respectively.  For this we have the following useful result.

\begin{theorem}\cite{Gha-Gud-1}\label{theorem-(2,1)}
A complex-valued function $\phi:(M,g)\to\cn$ from a Riemannian manifold is a $(2,1)$-harmonic morphism if and only if 
$$\kappa(\phi,\phi)=0,\ \ \tau^2(\phi)=0\ \ \text{and}\ \ \tau^2(\phi^2)=0.$$
\end{theorem}

\begin{lemma}\label{lemma-phi-phi}
Let $\phi:(M,g)\to\cn$ be a horizontally conformal proper biharmonic function.  Then the tension field $\tau(\phi^2)$ and the bitension field $\tau^2(\phi^2)$ of $\phi^2$ satisfy
\begin{itemize}
\item[(a)] $\tau(\phi^2)=2\cdot\tau(\phi)\,\phi\neq 0$,
\item[(b)] $\tau^2(\phi^2)=2\,(\tau(\phi)^2+2\cdot\kappa(\tau(\phi),\phi))$.
	\end{itemize}
\end{lemma}

\begin{proof}
Since the complex-valued function $\phi$ is horizontally conformal and biharmonic we know that $\kappa(\phi,\phi)=0$ and $\tau^2(\phi)=0$.  Then the result follows from the following elementary calculations.
\begin{eqnarray*}
\tau^2(\phi^2)
&=&\tau(\tau(\phi^2))\\
&=&\tau(\tau(\phi)\,\phi+2\,\kappa(\phi,\phi)+\phi\,\tau(\phi))\\
&=&2\cdot\tau(\tau(\phi)\,\phi)\\
&=&2\cdot\big\{\tau^2(\phi)\,\phi
+2\,\kappa(\tau(\phi),\phi)+\tau(\phi)^2\big\}.
\end{eqnarray*}
\end{proof}

\begin{definition}\label{definition-diff-operators-2}
Let $\lambda,f:M\to\rn$ be differentiable functions on an open subset $M$ of $\rn^m$ with coordinates ${\bf x}=(x,y,x_3,\dots,x_m)$. Then we define the {\it non-linear} partial differential operators $D_3,D_4$ by 
\begin{eqnarray*}
D_3(\lambda,f)
&=&
2\cdot (f_{x}^2-f_{y}^2)
-8\cdot (\lambda_{x}\,f_{x}-\lambda_{y}\,f_{y})
-4\cdot (f_{xx}-f_{yy}),
\end{eqnarray*}
\begin{eqnarray*}
D_4(\lambda,f)
&=&
4\cdot f_{x}f_{y}
-8\cdot (\lambda_{x}\,f_{y}+\lambda_{y}\,f_{x})
-8\cdot f_{xy}.	
\end{eqnarray*}
\end{definition}

\begin{proposition}
Let $\Phi:(\rn^m,g)\to\cn$ be the horizontally conformal submersion $$\Phi:{\bf x}\mapsto (x+iy)\cong (x\cdot\text{\bf e}_1+y\cdot\text{\bf e}_2)$$ 
with dilation $e^\lambda:\rn^m\to\rn^+$. If $\Phi$ is proper biharmonic then $\tau(\Phi^2)\neq 0$ and the bitension field $\tau^2(\Phi^2)$ of $\Phi^2$ satisfies
\begin{eqnarray*}
\tau^2(\Phi^2)
&=& 
 e^{4\lambda}\cdot D_3(\lambda,f)\cdot\text{\bf e}_1
+e^{4\lambda}\cdot D_4(\lambda,f)\cdot\text{\bf e}_2.
\end{eqnarray*}
\end{proposition}
\smallskip 

\begin{proof}
It follows from Lemma \ref{lemma-phi-phi} that $\tau(\Phi^2)\neq 0$ and  $$\tau^2(\Phi^2)=2\,\tau(\Phi)^2+4\,\kappa(\tau(\Phi),\Phi).$$ Then the following computations provide the result.
\begin{eqnarray*}
\tau(\Phi)^2&=&
\,e^{4\lambda}\cdot\{
((f_{x_1})^2-(f_{x_2})^2)\,\text{\bf e}_1
+2\,f_{x_1}f_{x_2}\,\text{\bf e}_2\}.
\end{eqnarray*}
\begin{eqnarray*}
\kappa(\tau(\Phi),\Phi)
&=&
\sum_{k=1}^m
X_k(\tau(\Phi))\cdot X_k(\Phi)\\
&=&
-\sum_{k=1}^2
X_k( e^{2\lambda}\,\{
f_{x_1}\,\text{\bf e}_1+f_{x_2}\,\text{\bf e}_2
\})\cdot e^{\lambda}\,\text{\bf e}_k\\
&=&
-\,e^{4\lambda}\big\{\sum_{k=1}^2
(2\,\lambda_{x_k}\,f_{x_1}+f_{x_1x_k})\,\text{\bf e}_1\cdot\text{\bf e}_k\\
& &\qquad +\sum_{k=1}^2
(2\,\lambda_{x_k}\,f_{x_2}+f_{x_2x_k})\,\text{\bf e}_2\cdot\text{\bf e}_k\big\}\\
&=&
-\,e^{4\lambda}
(2\,\lambda_{x_1}\,f_{x_1}+f_{x_1x_1}
-2\,\lambda_{x_2}\,f_{x_2}-f_{x_2x_2})
\,\text{\bf e}_1\\
& &
-\,e^{4\lambda}\,
(2\,\lambda_{x_2}\,f_{x_1}+f_{x_1x_2}
+2\,\lambda_{x_1}\,f_{x_2}+f_{x_2x_1})
\,\text{\bf e}_2.
\end{eqnarray*}

\end{proof}

\section{Explicit $(2,1)$-Harmonic Morphisms}

In Sections \ref{section-tension-fields} and  \ref{section-tension-fields-2},  we have defined the partial differential operators $D_1,D_2,D_3$ and $D_4$.  We will now use these to construct explicit proper $(2,1)$-harmonic morphisms $\Phi:(M,g)\to\cn$.  We will then show how these can be employed to produce a large variety of concrete proper biharmonic maps.

\begin{example}\label{example-(2,1)-1}
For an open subset $M$ of $\rn^m$ and  $A,B,\alpha,\beta\in\Omega(M)$ let the functions $\lambda,f:M\to\rn$ be defined by 
$$
\lambda ({\bf x})=\alpha,\ \ 
f({\bf x})= -2\log(A\cdot(\cos(t)\cdot x +\sin(t)\cdot y)+ B)+\beta.$$
If $A\neq 0$, then the associated horizontally conformal map $\Phi:M\to\cn$ is a proper $(2,1)$-harmonic morphism i.e. 
$$(f_x,f_y)\neq 0,\ 
D_1(\lambda,f)=0,\ D_2(\lambda,f)=0,\ 
D_3(\lambda,f)=0,\ D_4(\lambda,f)=0.$$
\end{example}

The next result is a reformulation of Proposition 3.9 of \cite{Gha-Gud-1}, see also Corollary 3.1 of \cite{Gha-Ou-1}. Together with Example \ref{example-(2,1)-1} it is a useful tool for manufacturing a large variety of proper $(2,1)$-harmonic morphisms $(M,g)\to N^2$, to Riemann surfaces on the non-trivial manifolds constructed there.

\begin{proposition}
\label{proposition-holomorphic-composition}
Let $(M,g)$ be a Riemannian manifold, $N^2$ be a Riemann surface and $\phi:M\to\cn$ be a proper $(2,1)$-harmonic morphism. Further, let $F:U\to N^2$ be a non-constant  holomorphic function defined on an open subset of $\cn$ containing $\phi (M)$.  Then the composition $F\circ \phi:(M,g)\to N^2$ is a proper $(2,1)$-harmonic morphism, in particular a proper biharmonic map.
\end{proposition}

Every complex-valued harmonic function, locally defined in the plane $\cn$, is the sum of a holomorphic and an anti-holomorphic one.  This leads us to the next statement.

\begin{proposition}\label{proposition-harmonic-composition}
Let $(M,g)$ be a Riemannian manifold and $\phi:M\to\cn$ be a submersive  $(2,1)$-harmonic morphism. Further, let $F,G:U\to\cn$ be holomorphic functions defined on an open subset $U$ of $\cn$ containing $\phi (M)$ and $\psi=F+\bar G$.  Then the composition $\psi\circ \phi:(M,g)\to\cn$ is a biharmonic map.  It is proper if and only if 
$$F_z(\tau(\phi))+\overline{G_{z}(\tau(\phi))}\neq 0.$$
\end{proposition}

\begin{proof}
The tension field $\tau$ is linear so it follows from Proposition \ref{proposition-holomorphic-composition} that the composition $\psi\circ\phi:U\to\cn$ is biharmonic.  Following the well-known composition law, see Corollary 3.3.13 of \cite{Bai-Woo-book}, we have
\begin{eqnarray*}
\tau(\psi\circ\phi)
&=&
d\psi(\tau(\phi))+\trace \nabla d\psi(d\phi,d\phi).
\end{eqnarray*}
The map $\phi:(M,g)\to\cn$ is horizontally conformal with dilation of the form $e^\lambda:M\to\rn^+$.  The standard basis vectors ${\bf e}_1$ and ${\bf e}_2$ form a global orthonormal frame on the open subset $U$ of $\cn$.  Let $X$ and $Y$ be their horizontal lifts via $\phi$ so that the vector fields $e^{-\lambda}X$ and $e^{-\lambda}X$ form an orthonormal frame for the horizontal distribution $\H$ of the tangent bundle $TM$. Then 
\begin{eqnarray*}
\trace \nabla d\psi(d\phi,d\phi)
&=&
\nabla d\psi(d\phi(e^{-\lambda}X),d\phi(e^{-\lambda}X))\\
& &
\quad\quad +\nabla d\psi(d\phi(e^{-\lambda}Y),d\phi(e^{-\lambda}Y))\\
&=&
e^{-2\lambda}\cdot 
(\nabla d\psi({\bf e}_1,{\bf e}_1)
+\nabla d\psi({\bf e}_2,{\bf e}_2))\\
&=&
e^{-2\lambda}(\tau(F+\bar G))\\
&=&
0.
\end{eqnarray*}
Furthermore we have
\begin{eqnarray*}
d\psi(\tau(\phi))
&=&
dF(\tau(\phi))+d\bar G(\tau(\phi))\\
&=&
F_z(\tau(\phi))+\overline{G_{z}(\tau(\phi))}.
\end{eqnarray*}
The stated result now follows from these calculations.
\end{proof}

\section{Acknowledgements}

The first author would like to thank the Department of Mathematics at Lund University for its great hospitality during her time there as a postdoc.

\end{document}